\documentclass[10pt]{amsart}
\usepackage[english]{babel}
\usepackage{amssymb}
\usepackage{amsmath}
\usepackage{srcltx}
\usepackage{lscape}

\newtheorem{dummy}{Dummy}

\newtheorem{theorem}[dummy]{Theorem}
\newtheorem{proposition}[dummy]{Proposition}
\newtheorem{corollary}[dummy]{Corollary}

\theoremstyle{definition}

\newtheorem{example}[dummy]{Example}

\newcommand{\ignore}[1]{}

\author{S. Pumpl\"un}
\email{susanne.pumpluen@nottingham.ac.uk}
\address{School of Mathematical Sciences\\
University of Nottingham\\
University Park\\
Nottingham NG7 2RD\\
United Kingdom
}

\begin{document}

\title[The automorphisms of differential extensions of characteristic $p$]{The automorphisms of differential extensions of characteristic $p$}

\begin{abstract}
 Nonassociative differential extensions are generalizations of associative differential extensions, either of a purely inseparable field extension $K$  of exponent one of a field $F$, $F$ of characteristic $p$, or of a central division algebra over a purely inseparable field extension of $F$. Associative differential extensions are well known central simple algebras first defined by Amitsur and Jacobson.
We explicitly compute the automorphisms of nonassociative differential extensions. These
 are canonically obtained by restricting automorphisms of the differential polynomial ring used in the construction of the algebra.
 In particular, we obtain descriptions for the automorphisms of associative differential extensions of $D$ and $K$, which are known to be inner.
\end{abstract}



\maketitle

\section{Introduction}\label{sec1}

Let $C$ be a field of characteristic $p>0$, $D$ a finite-dimensional division algebra over its center $C$ ($D$ may be $C$) and
$\delta$ a quasi-algebraic derivation on $D$, such that $\delta|_C$ has minimal polynomial $g(t)=t^{p^e}+a_1t^{p^{e-1}}+\dots+ a_et\in {\rm Const}(\delta)[t]$. When the differential polynomial $f(t)=g(t)-d\in R=D[t;\delta] $
generates a two-sided ideal $Rf$, then the quotient algebra $(D,\delta,d)=D[t;\delta]/(f)$ is a well-known associative central simple algebra over  the field  $F=C\cap {\rm Const}(\delta)$,  called a \emph{(generalized) differential extension of $D$} (we will usualy omit the term ``generalized'' here).
 Differential extensions  were originally defined and investigated by Amitsur  \cite{Am2},  Hoechsmann \cite{Hoe}, and Jacobson
  (for a recent comprehensive exposition cf. \cite[Sections 1.5, 1.8, 1.9]{J96}).
Indeed,  all associative central division algebras over a field $F$ of characteristic zero can
be constructed this way, using differential polynomials \cite{Am, Hoe}, see \cite[Sections 1.5, 1.8, 1.9]{J96}.

 Differential extensions $(D,\delta,d)$ were  generalized to nonassociative differential extensions of $D$  by employing differential polynomials $f(t)=g(t)-d\in D[t;\delta]$ in the construction that are not right-invariant (i.e., $Rf$ is not a two-sided ideal), using the following approach:
 Given $f\in D[t;\delta]$ of degree $m$, we equip the set of all differential polynomials of degree less than $m$ with a nonassociative ring structure  $g\circ h=gh \,\,{\rm mod}_r f $ using the remainder of dividing $gh$ on the right by $f$ to define the multiplication \cite{P16, P17}.
The resulting nonassociative unital ring (with the usual polynomial addition) denoted $(D,\delta,d)=D[t;\delta]/D[t;\delta]f$
is a nonassociative  algebra over $F$ \cite{P66}.
If $f\in D[t;\delta]$ is not right-invariant, then the right nucleus of $(D,\delta,d)$ is the eigenring of $f$ employed in \cite{Am} and \cite{O}.
 In particular, when $D=C$, then $\delta$ has the minimum polynomial $g(t)=t^p-t$ and
 $f(t)=t^p-t-d\in C[t;\delta]$ is irreducible, then
 $(C,\delta,d)=C[t;\delta]/C[t;\delta]f$ is a  unital nonassociative division algebras over $F$ of dimension $p^{2}$ for
 all $d\in C\setminus F$. This type of algebra  canonically generalizes Amitsur's associative cyclic extensions of degree $p$ and it was shown already in \cite{P17} that its automorphism group indeed also has a cyclic subgroup of order $p$.

 In this paper, we explicitly describe the automorphisms of the nonassociative differential extensions  $(D,\delta,d)$ and $(K,\delta,d)$, where  $D$ a finite-dimensional central division algebra
 over $C$, and  $K/F$ is a purely inseparable field extension  of exponent one. Results on the automorphisms of their associative ``cousins'' are obtained as well in the process.

 After  recalling the definitions of nonassociative differential extensions $(D,\delta,d)$ of a central simple division algebra  $D$ and of
 nonassociative differential extensions $(K,\delta,d)$ of a field $K$ in Section \ref{sec:prel},
 we investigate the automorphisms of $(D,\delta,d)$ in Section \ref{sec:2} and those of $(K,\delta,d)$ in Section  \ref{sec:3}.

The associative  differential extensions $(D,\delta,d)$ and $(K,\delta,d)$  are central simple over their center $F$, so all their algebra automorphisms are inner. We show that all automorphisms of the algebras  $(D,\delta,d)$ and  $(K,\delta,d)$ are canonically induced by certain automorphisms in  ${\rm Aut}(D[t;\sigma])$, respectively ${\rm Aut}(K[t;\sigma])$, and are of the form
 $$H_{\tau,c,e}( \sum_{i=0}^{p^e-1} d_i t^i )=\sum_{i=0}^{p^e-1}\tau( d_i ) (et+c)^i$$
  for suitable $e\in C^\times$, $c\in D$, respectively, of the form
 $H_{\tau,c,e}( \sum_{i=0}^{p^e-1} d_i t^i )=\sum_{i=0}^{p^e-1}\tau( d_i ) (et+c)^i$ for suitable $c\in K$, $e\in K^\times$. In particular,
  they extend an automorphism $\tau$ of $D$, respectively of $K$ (Theorems \ref{thm:aut1}, \ref{thm:aut1K}). Furthermore,
  $\{H_{\tau,0,1}\,|\, \tau\in {\rm Aut}_F(D) \text{ where } \delta\circ \tau=\tau\circ \delta \text{ and } d\in {\rm Fix}(\tau) \}$
 is a subgroup of ${\rm Aut}((D,\delta,d))$ (Corollary \ref{cor:7}).
 Moreover, the set of all logarithmic derivatives $\{ \delta(a)/a\,|\, a\in C \}$ and within that group,
 $\langle H_{id,-1}\rangle$, are subgroups of ${\rm Aut}((D,\delta,d))$ and of ${\rm Aut}((K,\delta,d) )$. The latter is a cyclic subgroup of order $p$ which leaves $D$, respectively $K$, invariant (Proposition \ref{prop:3}).
 The subgroup $\{H_{i_a, a^{-1}\delta(a)}\,|\, a\in {\rm Nuc}((D,\delta,d))^\times\}$
 of ${\rm Aut}_F((D,\delta,d) )$ is a subgroup of inner automorphisms (Theorem \ref{thm:main inner}), and the set of all logarithmic derivatives $\{\delta(a)/a\,|\, a\in K \}$ is a subgroup of inner automorphisms of $(K,\delta,d)$ (Theorem \ref{thm:main4}).
 All of this also applies when  $(D,\delta,d)$, respectively $(K,\delta,d)$, is associative.

The results presented here complement the ones on automorphisms of nonassociative cyclic and nonassociative generalized cyclic algebras $(K/F,\sigma,d)$ and $(D,\sigma,d)$ in arbitrary characteristic in \cite{BP19, P21}.


\section{Preliminaries} \label{sec:prel}

\subsection{Nonassociative algebras} \label{subsec:nonassalgs}


Let $F$ be a field. An $F$-vector space $A$ is called an
\emph{algebra} over $F$, if there is an algebra multiplication $\cdot$ defined on it, i.e. a map $A\times
A\to A$, $(x,y) \mapsto x \cdot y$ which is $F$-bilinear. The  multiplication $x \cdot y$ in $A$ is usually
denoted simply by juxtaposition. An algebra $A$ is called
\emph{unital} if there is an element in $A$, denoted by 1, such that
$1x=x1=x$ for all $x\in A$. We will only consider unital algebras.

An algebra $A\not=0$ is called a \emph{division algebra} if for any
$a\in A$, $a\not=0$, the left multiplication  with $a$, $L_a(x)=ax$,
and the right multiplication with $a$, $R_a(x)=xa$, are bijective.
If $A$ has finite dimension over $F$, $A$ is a division algebra if and only if $A$ has no zero divisors \cite[pp. 15, 16]{Sch}.

Associativity in $A$ is measured by the {\it associator} $[x, y, z] =
(xy) z - x (yz)$. The {\it left nucleus} ${\rm
Nuc}_l(A) = \{ x \in A \, \vert \, [x, A, A]  = 0 \}$,  {\it
middle nucleus}  ${\rm Nuc}_m(A) = \{ x \in A \, \vert \,
[A, x, A]  = 0 \}$, {\it right nucleus}
${\rm Nuc}_r(A) = \{ x \in A \, \vert \, [A,A, x]  = 0 \}$ and  {\it nucleus}
 ${\rm Nuc}(A) = \{ x \in A \, \vert \, [x, A, A] = [A, x, A] = [A,A, x] = 0 \}$ are associative
subalgebras of $A$. The nucleus contains $F1$
and $x(yz) = (xy) z$ whenever one of the elements $x, y, z$ is in ${\rm Nuc}(A)$. The
 {\it center} of $A$ is ${\rm C}(A)=\{x\in \text{Nuc}(A)\,|\, xy=yx \text{ for all }y\in A\}$.

 For a subring $B$ of a unital ring $A$, the \emph{centralizer} (also called the  \emph{commutator subring} if $A$ is associative) of $B$ in $A$ is defined as
${\rm Cent}_A(B)=\{a\in A\,|\, ab=ba \text{ for all } b\in B\}$.
If ${\rm Cent}_A(B)=B$ then ${\rm Cent}_A(B)$ is a maximal commutative nonassociative subring of $A$.

An element $0\not=a\in A$ has a \emph{left inverse} $a_l\in A$, if
$L_{a_l}(a)=a_l a=1$, and a \emph{right inverse}
 $a_r\in A$, if $R_{a_r}(a)=a a_r=1$. If  $a_r=a_l$ then we denote this element by
$a^{-1}$.

An automorphism $G\in {\rm Aut}_F(A)$ is called an \emph{inner automorphism}
if there is an element $a\in A$ with left inverse $a_l$ such
that $G(x) =G_a(x)= (a_lx)a$ for all $x\in A$.
The set of inner automorphisms $\{G_a\,|\, a\in {\rm Nuc}(A) \text{ invertible} \}$ is a subgroup of ${\rm Aut}_F(A)$ (\cite{W09}, \cite[Proposition 5.7]{CB}).
Given an inner automorphism $G_a\in {\rm Aut}_F(A)$ and any $H\in {\rm Aut}_F(A)$, $H^{-1}\circ G\circ H$ is the inner automorphism $G_{H(a)}(x)= (H^{-1}(a_l)x)H(a)$ \cite[p.~69]{CB}.

\subsection{Differential polynomial rings}\label{subsec:diff}

Let $D$ be an associative division ring. A \emph{derivation} on $D$ is an
additive map $\delta:D\rightarrow D$ such that $\delta(ab)=a\delta(b)+\delta(a)b$ for all $a,b\in D$.
 The \emph{differential polynomial ring} $D[t;\delta]$
is the set of polynomials $a_0+a_1t+\dots +a_nt^n$ with $a_i\in D$,
where addition is defined term-wise and multiplication by $ta=at+\delta(a)$ for all $a\in D.$
$R=D[t;\delta]$ is a left and right principal ideal domain and there is a right division algorithm in $R$: for all
$g,f\in R$, $g\not=0$, there exist unique $r,q\in R$ with ${\rm
deg}(r)<{\rm deg}(f)$, such that $g=qf+r$
 \cite{ J96, P66}.

  We call $f\in R$ a \emph{right-invariant} polynomial, if $fR\subset Rf$.
  If $f$ is right invariant then $Rf$ is a two-sided ideal.

 For $f=a_0+a_1t+\dots +a_nt^n$ with $a_n\not=0$ define ${\rm
deg}(f)=n$ and ${\rm deg}(0)=-\infty$. Then ${\rm deg}(fg)={\rm deg}(f)+{\rm deg}(g).$
 An element $f\in R$ is \emph{irreducible} in $R$
 if  it is no unit and it has no proper factors, i.e there do not exist
 $g,h\in R$, neither a unit, such that $f=gh$ \cite[p.~11]{J96}.

If $D$ is of characteristic $p$ and $R=D[t;\delta]$, define
$V_p(b)=b^p+\delta^{p-1}(b)+*$
 for all $b\in D$, with $*$ a sum of commutators of $b$, $\delta(b),\dots,$ and $ \delta^{p-2}(b)$, that is
$V_2(b)=b^2+\delta(b),$ $ V_3(b)=b^3+\delta^{2}(b)+[\delta(b),b],  $
 and so on  \cite[p.~18, (1.3.20)]{J96}.  If $D$ is commutative, or if $b\in D$ commutes with all its derivatives, then the sum $*$ in the equation is 0 and the formula simplifies to
$V_p(b)=b^p+\delta^{p-1}(b)$
 \cite[p.~17 ff.]{J96}.

  We have
$(t-b)^p=t^p-V_p(b)$
 for all $b\in D$ \cite[1.3.19]{J96}. Iterating this equation yields
$(t-b)^{p^e}=t^{p^e}-V_{p^e}(b)$
for all $b\in D$,  where
$V_{p^e}(b)=V_p^e(b)=V_p(\dots (V_p(b))\dots )$,
where the number of terms $V_p$ on the right-hand side is $e$ \cite[1.3.22]{J96}. For any $p$-polynomial
$f(t)=a_0t^{p^e}+a_1t^{p^{e-1}}+\dots+a_et-d\in D[t;\delta]$
we obtain
$$f(t)-f(t-b)=a_0V_{p^e}(b)+a_1V_{p^{e-1}}(b)+\dots+a_eb$$
for all $b\in D$ and define
$V_f(b)=a_0V_{p^e}(b)+a_1V_{p^{e-1}}(b)+\dots+a_eb.$

\subsection{Nonassociative algebras obtained from differential polynomial rings} \label{subsec:structure}

Let $R=D[t;\delta]$ and $f\in R$ of degree $m$. Let ${\rm mod}_r f$ denote the remainder of right division by $f$.
Define  ${\rm Cent}(\delta)=\{a\in D\,|\, \delta(a)=0\}$. 
 Then the vector space
$R_m=\{g\in D[t;\delta]\,|\, {\rm deg}(g)<m\}$
 together with the multiplication $g\circ h=gh \,\,{\rm mod}_r f ,$
where ${\rm mod}_r f$ denotes the remainder by right division with $f$, is a unital nonassociative ring that  we denote by $R/Rf$ or sometimes  $S_f$.  In the following, we denote the multiplication in these algebras  simply by juxtaposition. The ring $S_f$ is an algebra over
 $\{a\in D\,|\, ah=ha \text{ for all } h\in S_f\}$.
This is a commutative subring of $D$ \cite[(7)]{P66} and indeed, it is easy to check that ${\rm Cent}(\delta)\cap C(D)=\{a\in D\,|\, ah=ha \text{ for all } h\in S_f\}$. We write $F={\rm Cent}(\delta)\cap C(D)$  for this field. The algebra
 $S_f$ is associative if and only if $f$ is right-invariant. In this case, $S_f$ is the usual quotient algebra  $R/(f)$.

If $S_f$ is not associative then ${\rm Nuc}_l(S_f)={\rm Nuc}_m(S_f)=D$ and ${\rm Nuc}_r(S_f)=\{g\in R\,|\, fg\in Rf\}.$
 If $f\in R$ is irreducible and $S_f$ a finite-dimensional $F$-vector space
or free of finite rank as a right ${\rm Nuc}_r(S_f)$-module, then $S_f$
is a division algebra. Conversely, if $S_f$ is a division algebra then $f$ is irreducible.

We will assume from now on and throughout the paper that $C$ is a field  of characteristic $p$ and $D$  a finite-dimensional division algebra with center $C$ of degree $n$ (we allow $D=C$).
Let  $\delta$ be a derivation of $D$, such that $\delta|_C$ is algebraic with minimum polynomial
$g(t)=t^{p^e}+a_1t^{p^{e-1}}+\dots+ a_et\in {\rm Const}(\delta)[t]$
 of degree $p^e$.
  In particular, $\delta$ is a quasi-algebraic derivation on $D$ and so $R=D[t;\delta]$ is not simple \cite{LLLM},
 and $g(\delta)=I_{d_0}$ is an inner
derivation of $D$.  W.l.o.g. we may choose $d_0\in {\rm Const}(\delta)$ \cite[Lemma 1.5.3]{J96}.
 The center of $R=D[t;\delta]$  is $F[z]$ with
$z=g(t)-d_0$.
 For all $a\in C$, define a homomorphism $V:C\rightarrow F$ between additive groups via
  $V(a)=V_g(a)=V_{p^e}(a)+a_1V_{p^{e-1}}(a)+\dots+a_ea$  \cite{J37}. Then
$V(a)=0 \text{ if and only if }a=\delta(c)/c$
 for some $c\in C$ (\cite{J37}, cf. also \cite[p.~2]{Hoe}).

 Given $\tau\in {\rm Aut}(D)$ and $p(t)\in R$, the map
 $$H_{\tau,p(t)}:D[t;\delta]\rightarrow D[t;\delta], \, H_{\tau,p(t)}(\sum_{i=0}^{n} b_i t^i )= \sum_{i=0}^n\tau(a_i)p(t)^i$$
is a ring homomorphism if and only if
$p(t)\tau(z)=\tau(\sigma(z))p(t)+\tau(\delta(z))$
for all $z\in D$; and $H$ is bijective if and only if ${\rm deg} \, p(t)=1$  \cite[p. 4]{LL}.
  Indeed, for $p(t)=et+c$ with $c\in D$, $e\in D^\times$, $\tau \in {\rm Aut}(D)$, the
  map $H_{\tau,c,e}=H_{\tau,p(t)}$
is an automorphism of $D[t;\delta]$, if and only if
\begin{equation}\label{main}
c \tau(z)+e\delta(\tau(z))=\tau(z)c+\tau(\delta(z))
\end{equation}
 for all $z\in D$ and $e\in C^\times$.

 This means that  for $c\in D$, $e\in D^\times$, we have $H_{id,c,e}\in {\rm Aut}(R)$ if and only if
 $e\in C^\times$; $H_{\tau,0,1}\in {\rm Aut}(R)$ if and only if $\delta\circ\tau=\tau\circ\delta$; $H_{\tau,c,1}\in {\rm Aut}_F(R)$ if and only if $c\in C^\times$ \cite[Corollary 4.4]{CB}.
  In particular, if  $c\in C$ then $H_{\tau,c,e}$ is an automorphism of $R$, if and only if
 $\delta $ and $\tau$ commute \cite[Proposition 4.6]{CB}.

Any automorphism  $H:D[t;\delta]\rightarrow D[t;\delta]$ canonically induces an isomorphism
$S_f\cong S_{H(f)}.$
In particular, any $F$-automorphism  $H:D[t;\delta]\rightarrow D[t;\delta]$ canonically induces an isomorphism of $F$-algebras
$S_f\cong S_{H(f)}.$

\section{The automorphisms of nonassociative  differential extensions of $D$} \label{sec:2}

From now on we assume that $F$ is strictly contained in $C$.

For  $f(t)=g(t)-d\in D[t;\delta] $, the unital nonassociative $F$-algebra
$$(D,\delta,d)=D[t;\delta]/D[t;\delta] f(t)$$
is called a \emph{nonassociative (generalized) differential extension of $D$}, however, we will usually omit the term ``generalized in the following.
$(D,\delta,d)$ has center $F$ and dimension $p^{2e}n^2$.
For $d\in D\setminus F$ it has left and middle nucleus $D$.
For $d\in C\setminus F$ we have  ${\rm Nuc}_r((D,\delta,d))=D$. This follows from the corrected version of the proof of \cite[Lemma 19]{P17}, cf.
\cite[Lemma 21]{Psemi}, and implies that $D= {\rm Nuc}((D,\delta,d))$ for $d\in C\setminus F$.

For $d\in C\setminus F$, $(C,\delta|_C,d)$ is a subalgebra of $(D,\delta,d)$.
For all $d\in D$ and $a\in C$, $(D,\delta,d)\cong (D,\delta,d-V(a)).$

$(D,\delta,d)$ is associative if and only if $d\in F$ and a division algebra if and only if $f(t)$ is irreducible.
For $d\in F$, $(D,\delta,d)$ is a central simple algebra over $F$ \cite[p.~23]{J96}.
For associative algebras $(D,\delta,d)$, the centralizer of $C$ is $D$ \cite[Theorem 3.1]{Hoe}, and it is clear that $D\subset {\rm Cent}_{(D,\delta, d)}(C)=\{a(t)\in (D,\delta, d)\,|\, a(t)c=ca(t) \text{ for all } c\in C\}$ also in the nonassociative case.
\\
\begin{proposition} \label{prop:centralizer}
 Let $\delta$ have minimum polynomial $g(t)=t^p+a_1t$.
 For all $d\in D$, the centralizer ${\rm Cent}_{(D,\delta, d)}(C)$ of $C$  in $(D,\delta, d)$ is $D$.
\end{proposition}

\begin{proof}
It remains to check that ${\rm Cent}_{(D,\delta, d)}(C)\subset D$.
For each $\sum_{i=0}^{p-1}b_it^i\in {\rm Cent}_{(D,\delta, d)}(C)$, we know that $c(\sum_{i=0}^{p-1}b_it^i)=(\sum_{i=0}^{p-1}b_it^i)c$ for all $c\in C$, which is equivalent to
$$\sum_{i=0}^{p-1}c b_it^i = \sum_{i=0}^{p-1}b_i(t^ic)= \sum_{i=0}^{p-1} \sum_{j=0}^{i}\binom{i}{j} b_i \delta^{i-j}(c) t^j$$
for all $c\in C$.
Comparing coefficients we get
\\ for $t^{p-2}$: $cb_{p-2}=b_{p-1}\binom{p-1}{p-2}\delta(c)+b_{p-2}c$, hence $b_{p-1}\binom{p-1}{p-2}\delta(c)=0$ for all $c\in C$. This implies $b_{p-1}=0$, since we assumed that $\delta|_C\not=0$.
\\ For $t^{p-3}$: $cb_{p-3}=b_{p-1}\binom{p-1}{p-3}\delta^2(c)+b_{p-2}\binom{p-1}{p-2}\delta^3(c)+b_{p-3}c$,
hence $b_{p-2}\binom{p-1}{p-2}\delta^3(c)=0$ for all $c\in C$. This implies $b_{p-2}=0$, since we assumed that $\delta|_C\not=0$.
\\
Continuing this way we see that $b_i=0$ for all $i>0$ and so $\sum_{i=0}^{p-1}b_it^i=b_0\in D$, yielding ${\rm Cent}_{(D,\delta, d)}(C)\subset D$.
\end{proof}

We conjecture that for all nonassociative algebras $(D,\delta,d)$, the centralizer of $C$ is $D$ but were only able to explicitly calculate this for the special case $e=1$. The above proof clearly relies on comparing coefficients and the general case is probably an equally straightforward  tedious calculation.

All the automorphisms of a nonassociative differential extension $(D, \delta, d)$ are canonically induced by ring automorphisms of the differential polynomial ring $D[t;\delta]$:

\bigskip
\begin{theorem} \label{thm:aut1}
Let $(D, \delta, d)$ be a   differential extension of dimension $p^en^2$
over $F$, i.e. $d\in D\setminus F$. Let $H \in {\rm Aut}((D, \delta, d))$.
\\ (i)  There exists some $\tau\in {\rm Aut}(D)$ and $e\in C^\times$, $c\in D$, which satisfy $\tau(z)c+\tau(\delta(z))=c\tau(z)+e\delta(\tau(z))$
for all $z\in D$, such that $H=H_{\tau,c,e}$ with
\begin{equation} \label{eqn:neccessaryII}
H_{\tau,c,e}( \sum_{i=0}^{p^e-1} d_i t^i )=\sum_{i=0}^{p^e-1}\tau( d_i ) (et+c)^i.
\end{equation}
\\ (ii) If $\delta$ and $\tau$ commute on $C$ then $H=H_{\tau,c,1}$ for some $c\in C$.
\\ (iii)
 All maps $H_{\tau,c,e}$ where $e\in C^\times$,  $c\in D$,  $\tau \in {\rm Aut}_{F}(D)$  satisfy $\tau(z)c+\tau(\delta(z))=c\tau(z)+e\delta(\tau(z))$
for all $z\in D$, and
$H(f)=f$,  are automorphisms of $(D, \delta, d)$, and induced by automorphisms of $D[t;\delta]$.
\end{theorem}

This substantially generalizes  \cite[Theorem 5.10]{CB}.

\begin{proof}
Let first $(D, \delta, d)$ be a  differential extension
over $F$ which is \emph{not} associative, i.e. $d\in D\setminus F$.
\\ (i) Let $H \in {\rm Aut}((D, \delta, d))$. We know that $H$ leaves the left nucleus invariant, so  $H|_D\in {\rm Aut}(D)$. Thus $H|_D = \tau$ for some $\tau\in {\rm Aut}(D)$. Since $H$ is bijective it must preserve the degree of $t$.
This can also be tediously  computed directly by putting
$H(t) = \sum_{i=0}^{p^e-1} b_i t^i$ for some $b_i \in D$, and then comparing the coefficients of
$$H(tz) = H(t)H(z) = \big( \sum_{i=0}^{p^e-1} b_i t^i \big) \tau(z) = \sum_{i=0}^{p^e-1} \sum_{j=0}^{i}\binom{i}{j} b_i \delta^{i-j}(\tau(z)) t^j,$$
and
$$H(tz) = H(zt+\delta(z)) = \tau(z) \sum_{i=0}^{p^e-1} b_i t^i +\tau(\delta(z))$$
for  $z\in D$, using that $F$ is strictly contained in $C$. We thus have $H(t) = b_1t+b_0$ for $b_0,b_1\in D$.
 Comparing the terms in $H(tz) = (b_1t+b_0)H(z)$, we obtain
$$\tau(z)b_0+\tau(\delta(z))=b_0\tau(z)+b_1\delta(\tau(z)) \text{ and } b_1\tau(z) =\tau(z)b_1$$
for all $z\in D$, so $b_1\in C^\times$. In particular, we have $\tau(\delta(z))=b_1\delta(\tau(z))$
for all $z\in C$.
\\ (ii) If $\delta|_C$ and $\tau|_C$ commute  and since  $\delta|_C\not=0,$ this implies that $b_1=1$.
This means $\tau(z)b_0=b_0\tau(z)$ for all $z\in D$, so that $b_0=c\in C$.
\\ (iii)
All maps $H_{\tau,c,e}$ where $d\in C^\times$ and  $\tau \in {\rm Aut}(D)$  satisfies $\tau(z)c+\tau(\delta(z))=c\tau(z)+e\delta(\tau(z))$
for all $z\in D$  are automorphisms of $R$ and thus induce isomorphisms between the algebras $(D,\delta,d)$ and $S_H(f)$.
If $H(f)=f$ then $H_{\tau,c,e}\in {\rm Aut}((D,\delta,d))$.
\\  Let now $(D, \delta, d)$ be associative. We know that every algebra isomorphism $H\in {\rm Aut}_{F}((D,\delta,d))$ leaves the centralizer ${\rm Cent}_{(D,\delta, d)}(C)=D$ invariant, so  $H|_D\in {\rm Aut}_{F}(D)$. Thus again $H|_D = \tau$ for some $\tau\in {\rm Aut}_{F}(D)$. Using the same proofs as in (i), (ii), (iii), the assertions also hold for an associative  differential extension.
\end{proof}

 Denote $H_{\tau,c,1}=H_{\tau,c}$. Then
 $$H_{id,-c}(\sum_{i=0}^{p^{e}-1} b_it^i)=\sum_{i=0}^{p^{e}-1} b_i(t-c)^i$$ is an automorphism of $(D,\delta,d)$ for all
$c\in C$ with $V(c)=0$,
and
$$\{c\in C\,|\, V(c)=0 \}=\{\delta(a)/a\,|\, a\in C \}\cong\{H_{id,-c}\,|\, c\in C \text{ with }
V(c)=0\}$$
is a subgroup of ${\rm Aut}((D,\delta,d) )$ via $c\mapsto H_{id,-c}$
  \cite[Propositions 8, 9]{P17}.

\bigskip
  \begin{proposition}\label{prop:3}
 (i) The cyclic subgroup
  $\langle H_{id,-1}\rangle $ of ${\rm Aut}((D,\delta,d))$ has order $p$ and leaves $D$ invariant.
\\  (ii)  If $c\in C$ with $V(c)=0$,  then $H_{id,-c}\in {\rm Aut}((D,\delta,d))$ has order $p$ and leaves $D$ invariant.
  \\ (iii) If $\delta\circ \tau=\tau\circ \delta$ and $\tau\in {\rm Aut}(D)$ has order $m$, then  $H_{\tau,1}\in {\rm Aut}(R)$ has order $m$.
 \end{proposition}

\bigskip
  The proof is trivial, observing that
  $$(H_{id,-c})^j(\sum_{i=0}^m a_it^i)=\sum_{i=0}^m a_i(t-jc)^i$$
   for an integer $j\geq 1$, and that
   $V(1)=0$ if and only if $a=\delta(a)$ for some $a\in C$.

\bigskip
\begin{example}
  If $g(t)=t^p-t$ is the minimal polynomial of $\delta|_C$ and $f(t)=t^p-t-d\in D[t;\delta]$, we have  $H_{id,-a}(f(t))=f(t)$ for all $a\in C$ with $\delta^{p-1}(a)+a^p-a=0$, and $\langle H_{id,-1}\rangle $ is a cyclic subgroup of ${\rm Aut}((D,\delta,d))$ of order $p$ which leaves $D$ invariant \cite[Lemma 10]{P17}.
If   $f(t)=t^p-t-d\in C[t]$, then
$(C,\delta|_C,d)$ is a subalgebra of $(D,\delta,d)$ and $H_{id, -1}|_C$ also generates a cyclic subgroup of ${\rm Aut}((C,\delta|_C,d))$. If $f$ is additionally irreducible and $d\in F$, then this is why the associative division algebra $(D,\delta,d)$ is also called a \emph{cyclic extension of $D$ of degree $p$} by Amitsur \cite{Am2}.
\end{example}

\bigskip
\begin{theorem}\label{thm:main1}
Let $H_{\tau,-c}\in {\rm Aut}(R)$,
 such that $V_g(c)=d-\tau(d)$. Then $H_{\tau,-c}$
induces an $F$-automorphism of $(D,\delta,d)$. In particular, if $\tau$ has order $m$ and $d\in {\rm Fix}(\tau)$, then $H_{\tau,0}\in {\rm Aut}((D,\delta,d))$ generates a subgroup of order $m$.
\end{theorem}

\begin{proof}
We know that $H_{\tau,-c}\in {\rm Aut}(R)$ implies that
$\tau\in {\rm Aut}_F(D)$ and $\delta(\tau(z))-\tau(\delta(z))=c\tau(z)-\tau(z)c$ for all $z\in D$. Moreover,
 $H_{\tau,-c}(f)=(t-c)^{p^e}+a_1(t-c)^{p^{e-1}}+\dots+ a_e(t-c)-\tau(d)=g(t-c)-\tau(d)=
  g(t)-V_g(c)-\tau(d)=f(t)+d-V_g(c)-\tau(d)$.
 Hence if $V_g(c)=d-\tau(d)$, then $H_{\tau,-c}\in {\rm Aut}((D,\delta,d))$.
 In particular, $V_g(0)=0=d-\tau(d)$ is equivalent to $d=\tau(d)$.
  \end{proof}

\bigskip
\begin{corollary}
Choose  $c\in C$ and assume that $\delta\circ \tau=\tau\circ \delta$.
Suppose one of the following holds:
\\ (i)  $d\in {\rm Fix}(\tau)$ (e.g., $(D,\delta,d)$ is associative) and $V(c)=0$.
 \\ (ii)   $V_g(c)=d-\tau(d)$.
 \\ Then  $H_{\tau,-c}\in {\rm Aut}(R)$
induces an automorphism of $(D,\delta,d)$.
\end{corollary}

\begin{proof}
By assumption, we have
 $\delta(\tau(z))-\tau(\delta(z))=c\tau(z)-\tau(z)c=0$ for all $z\in D$, that is $H_{\tau,-c}\in {\rm Aut}(R)$.
We also know $(D,\delta, d)\cong S_{H(f)}$ and
$$H_{\tau,-c}(f)=(t-c)^{p^e}+a_1(t-c)^{p^{e-1}}+\dots+ a_e(t-c)-\tau(d)$$
$$=g(t-c)-\tau(d)=g(t)-V_g(c)-\tau(d)=f(t)+d-V_g(c)-\tau(d).$$
 Thus if $V(c)=d-\tau(d)$ then $H_{\tau,-c}$ is an isomorphism.
The assertions now follow.
\end{proof}

\bigskip
\begin{corollary}\label{cor:7}
 Assume that $\tau\in {\rm Aut}(D)$ such that $\delta\circ \tau=\tau\circ \delta$ and $d\in {\rm Fix}(\tau)$.
  \\ (i) We have $H_{\tau,0}\in {\rm Aut}((D,\delta,d))$, i.e
$$\{H_{\tau,0}\,|\, \tau\in {\rm Aut}(D) \text{ such that } \delta\circ \tau=\tau\circ \delta \text{ and } d\in {\rm Fix}(\tau) \}$$
 is a subgroup of ${\rm Aut}_F((D,\delta,d))$.
In particular, if $(D,\delta,d)$ is associative then
$$\{H_{\tau,0}\,|\, \tau\in {\rm Aut}(D) \text{ such that } \delta\circ \tau=\tau\circ \delta  \}$$
 is a subgroup of ${\rm Aut}((D,\delta,d))$.
\\ (ii)  For all $c\in C$, $id_D\in {\rm Aut}(D)$ extends to an automorphism $H=H_{id,-c}\in {\rm Aut}((D,\delta,d))$.
\end{corollary}

\begin{proof}
(i) The first statement follows directly from Theorem \ref{thm:main1} using the map
  $\tau \mapsto H_{\tau,0}$.
 If $(D,\delta,d))$ is associative then $d\in F\subset {\rm Fix}(\tau)$ so $H_{\tau,0}\in {\rm Aut}((D,\delta,d))$ for all
  $\tau \in {\rm Aut}((D,\delta,d))$ hence the second assertion is clear, too.
  \\ (ii)  We know that $H_{id,-c}\in {\rm Aut}(R)$ canonically induces an isomorphism $H_{id,-c}\in {\rm Aut}((D,\delta,d))$ if
$V(c)=d-\tau(d)$, which is satisfied here.
\end{proof}

For all invertible $a\in {\rm Nuc}((D,\delta,d))$, the maps $G_a:(D,\delta,d)\rightarrow (D,\delta,d)$,
$$G_a(h)=(a^{-1}h)a,$$
 are inner automorphisms of
$(D,\delta,d)$ \cite[Corollary 5.8]{CB} and form a subgroup of ${\rm Aut}_F((D,\delta,d) )$. We know that  ${\rm Nuc}((D,\delta,d))\subset D$.
It is clear that for $a\in {\rm Nuc}((D,\delta,d))^\times$, $G_a|_D=i_a\in {\rm Aut}_C(D)\subset  {\rm Aut}_F(D)$ is inner, too.

\bigskip
\begin{theorem}\label{thm:main inner}
Let $a\in  {\rm Nuc}((D,\delta,d))^\times\subset D^\times$ and $G_a\in {\rm Aut}_F((D,\delta,d))$,
then
$$G_a=H_{i_a, a^{-1}\delta(a)}$$
 and
$\{H_{i_a, a^{-1}\delta(a)}\,|\, a\in {\rm Nuc}((D,\delta,d))^\times\}$  is a  subgroup of ${\rm Aut}_F((D,\delta,d) )$ of inner automorphisms.
In particular,  $\{a^{-1}\delta(a) \,|\, a\in C\cap  {\rm Nuc}((D,\delta,d))^\times \}$
is a subgroup of  $\{H_{i_a, a^{-1}\delta(a)}\,|\, a\in {\rm Nuc}((D,\delta,d))^\times\}$.
\end{theorem}

\begin{proof}
Since $G_a\in {\rm Aut}_F((D,\delta,d))$, $G_a=H_{\tau,c,e}$ for suitable $\tau\in {\rm Aut}_{F}(D)$ and $c\in D$, $e\in C^\times$, which immediately implies that $\tau=G_a|_D$, $\tau(x)=a^{-1}xa=i_a(x)$. Now since $G_a(t)=a^{-1}ta=t+a^{-1}\delta(a)$ we get $G_a=H_{i_a, a^{-1}\delta(a)}$.
This implies  that $\{H_{i_a, a^{-1}\delta(a)}\,|\, a\in D^\times\}$
is a subgroup of ${\rm Aut}_F((D,\delta,d) )$ of inner automorphisms. The second assertion is then clear.
\end{proof}

When $(D,\delta,d)$ is associative,
 obviously every  $H_{\tau,c,e}\in {\rm Aut}_F((D,\delta,d))$ is inner.

\section{The automorphisms of nonassociative  differential extensions of $K$} \label{sec:3}

What happens when $D=C$? Most of the results of the previous section apply:

Let $K$ be a field of characteristic $p$ together with an algebraic derivation
of $K$ of degree $p^e$ with minimum polynomial
$g(t)=t^{p^e}+a_1t^{p^{e-1}}+\dots+ a_et\in F[t]$, and $F={\rm Const}(\delta)$. Put $R=K[t;\delta]$. Then $K$  is a purely inseparable extension of $F$ of exponent one, $K=F(u_1, \dots,u_e)=F(u_1)\otimes_F\dots \otimes_F F(u_e)$, $u_i^p=a_i\in F$, $[K:F]=p^e$, and $K^p\subset F\subset K$.  The center of $R$ is $F[z]$ with $z=g(t)-d_0$, $d_0\in F$.
Let $f(t)=g(t)-d\in K[t;\delta]$, then the unital $F$-algebra
$$(K,\delta,d)=K[t;\delta]/K[t;\delta] f(t)$$
has dimension $p^{2e}$ and is called a \emph{nonassociative differential extension} of $K$. The algebra
$(K,\delta,d)$ is associative if and only if $d\in F$, and a division algebra if and only if $f(t)$ is irreducible.
For $d\in K\setminus F$ it has left and middle nucleus $K$ and the right nucleus contains $K$, thus
${\rm Nuc}((K,\delta,d))=K.$

For $d\in F$, $(K,\delta,d)$
is an associative central simple $F$-algebra \cite[p.~23]{J96} with maximal subfield $K$ and contains the subring $F[t]/(f(t))$, which is a maximal subfield of $(K,\delta,d)$, if $f(t)$ is irreducible in $F[t]$.

For associative algebras $(K,\delta,d)$, the centralizer of $K$
is $K$  \cite[Theorem 3.1]{Hoe}, and  $K\subset {\rm Cent}_{(K,\delta, d)}(K)$
also in the nonassociative case. Proposition \ref{prop:centralizer}, the centralizer ${\rm Cent}_{(K,\delta, d)}(K)$ of $K$  in $(K,\delta, d)$ is $K$ also when $d\in K\setminus F$, if $\delta$ has the minimum polynomial $g(t)=t^p+a_1t$ (and probably also in the general case for any $g(t)$).
Analogously  to Theorem \ref{thm:aut1} we have:

\bigskip
\begin{theorem} \label{thm:aut1K}
Let $(K, \delta, d)$ be a nonassociative   differential  extension of the field $K$ of dimension $p^{2e}$
over $F$.
Let $H \in {\rm Aut}((K, \delta, d))$.
\\ (i)  There exists some $\tau\in {\rm Aut}(K)$ which satisfies $\tau(\delta(z))=e\delta(\tau(z))$
for all $z\in K$ and $e\in K^\times$, such that $H=H_{\tau,c,e}$.
\\ (ii) If $\delta$ and $\tau$ commute, then $H=H_{\tau,c,1}$ for some $c\in K$.
\\ (iii)
 All maps $H_{\tau,c,e}$ where $e\in K^\times$, $\tau \in {\rm Aut}(K)$  satisfies $\tau(\delta(z))=e\delta(\tau(z))$
for all $z\in K$, and
$H(f)=f$,  are automorphisms of $(K, \delta, d)$, and induced by automorphisms of $K[t;\delta]$.
\end{theorem}

\begin{proof}
Let first $(K, \delta, d)$ be a   differential extension of $K$
over $F$ which is not associative, i.e. $d\in K\setminus F$.
\\ (i)  We know that $H$ leaves the left nucleus invariant, so $H|_K = \tau$ for some $\tau\in {\rm Aut}_{F}(K)$.  $H$ is bijective so must preserve the degree of $t$.
As in the proof of  Theorem \ref{thm:aut1}, we obtain $H(t)=(et+c)$ for $e\in K^\times$, $c\in K$ with
$\tau(\delta(z))=e\delta(\tau(z))$
for all $z\in K$.
\\ (ii) If $\delta$ and $\tau$ commute this implies that $e=1$.
\\ (iii)
All maps $H_{\tau,c,e}$ where $d\in C^\times$ and  $\tau \in {\rm Aut}_{F}(K)$  satisfies $\tau(\delta(z))=e\delta(\tau(z))$
for all $z\in D$  are automorphisms of $R$ and thus induce isomorphisms between the algebras $(K,\delta,d)$ and $S_H(f)$.
If $H(f)=f$ then $H_{\tau,c,e}\in {\rm Aut}((K, \delta, d))$.
\\  Let now $(K, \delta, d)$ be associative. We know that every algebra isomorphism $H\in {\rm Aut}((K,\delta,d))$ leaves the centralizer ${\rm Cent}_{(K,\delta, d)}(K)=K$ invariant, so  $H|_K\in {\rm Aut}(K)$. Thus again
the assertions in (i), (ii), (iii)  hold for an associative  differential extension $(K,\delta,d)$ using analogous proofs.
\end{proof}

From Proposition \ref{prop:3} we obtain:

\bigskip
  \begin{corollary}
 (i) The cyclic group $\langle H_{id,-1}\rangle $ is a subgroup of ${\rm Aut}((K,\delta,d))$ of order $p$ which leaves $K$ invariant.
\\  (ii)  If $c\in K$ with $V(c)=0$,
 then $H_{id,-c}\in {\rm Aut}((K,\delta,d))$ has order $p$.
 \\ (iii) If $\delta\circ \tau=\tau\circ \delta$ and $\tau\in {\rm Aut}(K)$ has order $m$, then  $H_{\tau,1}\in {\rm Aut}(R)$ has order $m$.
 \end{corollary}

\bigskip
Theorem \ref{thm:main1} for $D=C$ yields:

\bigskip
\begin{theorem}\label{thm:main2}
Let $H_{\tau,-c}\in {\rm Aut}_F(R)$,
 such that $V_g(c)=d-\tau(d)$. Then $H_{\tau,-c}$
induces an automorphism of $(K,\delta,d)$. In particular, if $\tau$ has order $m$ and $d\in {\rm Fix}(\tau)$, then $H_{\tau,0}\in {\rm Aut}_F((K,\delta,d))$ generates a subgroup of order $m$.
\end{theorem}

\bigskip
\begin{corollary}
Assume that $\delta\circ \tau=\tau\circ \delta$.
Suppose one of the following holds:
\\ (i)  $d\in {\rm Fix}(\tau)$ (e.g., $(K,\delta,d)$ is associative) and $V(c)=0$.
 \\ (ii)   $V_g(c)=d-\tau(d)$.
 \\ Then  $H_{\tau,-c}$ induces an $F$-automorphism of $(K,\delta,d)$.
\end{corollary}

\bigskip
\begin{corollary}\label{cor:7'}
(i) The set
$$\{H_{\tau,0}\,|\, \tau\in {\rm Aut}(K) \text{ such that } \delta\circ \tau=\tau\circ \delta \text{ and } d\in {\rm Fix}(\tau) \}$$
 is a subgroup of ${\rm Aut}((K,\delta,d))$.
In particular, if $(K,\delta,d)$ is associative then
$$\{H_{\tau,0}\,|\, \tau\in {\rm Aut}(K) \text{ such that } \delta\circ \tau=\tau\circ \delta  \}$$
 is a subgroup of ${\rm Aut}_F((K,\delta,d))$.
\\ (ii)  For all $c\in C$, $id_K\in {\rm Aut}(K)$ extends to an automorphism $H=H_{id,-c}\in {\rm Aut}((K,\delta,d))$.
\end{corollary}

\bigskip
 \begin{example}
If $g(t)=t^p-t\in F[t]$ is the minimum polynomial of $\delta$ and $f(t)=t^p-t-d\in K[t;\delta]$, then the automorphism group of $(K,\delta,d)$
has a cyclic subgroup of order $p$  generated by $H_{id,-1}$ which leaves $K$ invariant.
If $f(t)=t^p-t-d \in F[t]$ is irreducible then Amitsur called the division algebra
 $(K,\delta,d)$  a \emph{cyclic extension of $K$} of degree $p$ \cite{Am2} because  it can be seen as a noncommutative generalization of
 a cyclic field extension of $K$: it  has dimension $p$ as a $K$-vector space and
 the automorphism group of $(K,\delta,d)$ has a cyclic subgroup of order $p$.
 All associative cyclic extensions of $K$ of degree $p$ are of this form  \cite{Am2}.
Note that they always contain the cyclic separable field extension $F[t]/(t^p-t-d)$ of degree $p$, so that in the associative case,
these algebras are always cyclic.
\end{example}

\bigskip
Since ${\rm Nuc}((K,\delta, d))=K$, for all $0\not=a\in K$, $G_a(x)=a^{-1}x a$ is an inner automorphism of $(K,\delta,d)$
extending $id_K$. By the obvious generalizations of \cite[Lemma 2, Theorem 3, 4]{W09} to infinite fields,  $\{G_a\,|\, 0\not=a\in K \}$ is a non-trivial
subgroup of ${\rm Aut}_F((K, \delta, d))$. It is of course well-known that $\{G_a\,|\, 0\not=a\in {\rm Nuc}((K, \delta, d))=(K,\delta, d) \}={\rm Aut}_F((K,\delta, d))$ when $(K, \delta, d)$ is associative.
\\
\begin{theorem}\label{thm:main4}
Let $a\in K^\times$ and $G_a$ be an inner automorphism, then $G_a=H_{id, a^{-1}\delta(a)}$.
\end{theorem}

\begin{proof}
Since $G_a\in {\rm Aut}_F((K,\delta,d))$, $G_a=H_{\tau,c,e}$ and since $G_a(t)=a^{-1}ta=t+a^{-1}\delta(a)$ we get $G_a=H_{id, a^{-1}\delta(a)}$.
\end{proof}

The results of the previous sections, read in this context, thus mean that
$$\{\delta(a)/a\,|\, a\in K \}\cong\{H_{id,-c}\,|\, c\in K \text{ with } V(c)=0\}$$
is  a subgroup of ${\rm Aut}_F((K,\delta,d) )$ of inner automorphisms.
Within that subgroup,
 $\langle H_{id,-1}\rangle $ is a cyclic subgroup  of order $p$ which leaves $K$ invariant.
 Indeed, for all  $c\in K$ with $V(c)=0$, the inner automorphism $H_{id,-c}\in {\rm Aut}_F((K,\delta,d))$ has order $p$.
   All this holds in particular when $(K,\delta,d)$ is associative.

 \bigskip
 \begin{example}\label{ex:1}
Let  $x$ be an indeterminate and canonically extend the derivation $\delta$ on $K$  to $K(x)$ via setting
$\delta(x)=0$. Suppose that $\delta:K\rightarrow K$ has minimum polynomial $g(t)=t^p-t\in F[t]$
 Then ${\rm Const}(\delta)=F(x)$ and $g(t)=t^p-t\in F(x)[t]$
is the minimal polynomial of the extended derivation $\delta$. Let $A=(K(x),\delta,h(x))$ for $h(x)\in K(x)$.
 For $h(x)=x$, $(K(x),\delta,x)$ is an associative cyclic division algebra over $F(x)$ of degree $p$ \cite[Proposition 1.9.10]{J96}. For all $h(x)\in K(x)\setminus F(x)$,
$(K(x),\delta,h(x))$
 is a nonassociative cyclic division algebra over $F(x)$
of degree $p$  \cite[Example 16]{P17}. In both cases (associative and nonassociative), the $F(x)$-automorphisms of $A$ are canonically induced by the $F(x)$-automorphisms of $K(x)[t;\delta]$ and
consist of the maps $H_{\tau,c,e}$ where $e\in K(x)^\times$ and $\tau \in {\rm Aut}_{F(x)}(K(x))$  satisfy $\tau(\delta(z))=e\delta(\tau(z))$
for all $z\in K(x)$. Furthermore,
 $\{H_{id,-c}\,|\, c\in K(x) \text{ with } V(c)=0\}$
is a subgroup of ${\rm Aut}_{F(x)}(A)$ of inner automorphisms and
 $\langle H_{id,-1}\rangle $ is a cyclic subgroup of order $p$ within this one  which leaves $K(x)$ invariant.

 \end{example}

\emph{Acknowledgements}  This paper was written while the author was a visitor at the University of Ottawa. She acknowledges support from  the Centre de Recherches Math\'ematiques for giving a colloquium talk,
  and from Monica Nevins' NSERC Discovery Grant RGPIN-2020-05020. The author herself does not hold any grant. She would like to thank the Department of Mathematics and Statistics for its hospitality. She would also like to thank the anonymous referee for the kind report.



\begin{thebibliography}{[B-C-R]}

\bibitem{Am} A. S. Amitsur, \emph{Differential polynomials and division algebras.}
 Annals of Mathematics, Vol. 59  (2) (1954), 245-278.

\bibitem{Am2} A. S. Amitsur, \emph{Non-commutative cyclic fields.}
Duke Math. J. 21 (1954), 87-105.

\bibitem{BP19}  C. Brown and S. Pumpl\"un, \emph{Nonassociative cyclic extensions of fields and central simple algebras}.
 J. Pure Applied Algebra 223 (6) 2019, 2401-2412.\\
\verb#https://doi: 10.1016/j.jpaa.2018.08.018#


\bibitem{CB} C. Brown, \emph{Petit algebras and their automorphisms}. PhD Thesis, University of Nottingham 2018.\\
	\verb#https://eprints.nottingham.ac.uk/49613/#


\bibitem{Hoe} K. Hoechsmann,  \emph{Simple algebras and derivations.} Trans. Amer. Math. Soc. 108 (1963), 1-12.\\
\verb#https://doi.org/10.2307/1993822#


\bibitem{J96} N.~Jacobson,
``Finite-dimensional division algebras over fields.'' Springer Verlag,
Berlin-Heidelberg-New York, 1996.

\bibitem{J37} N.~Jacobson,  \emph{Abstract derivation and Lie algebras}. Trans. Amer. Math. Soc. 42 (2) (1937),
 206-224.



\bibitem{LL} T.-Y.~Lam and A.~Leroy, \emph{Homomorphisms between Ore extensions.}
Comtemporary Mathematics, 124 (1992), 83-110.

\bibitem{LLLM} T. Y. Lam, K. H. Leung, A. Leroy and J. Matczuk,
 \emph{Invariant and semi-invariant polynomials in skew polynomial rings}.
 Ring theory 1989 (Ramat Gan and Jerusalem, 1988/1989), 247-261,
 Israel Math. Conf. Proc., 1, Weizmann, Jerusalem, 1989.

\bibitem{O} O. Ore, \emph{Formale Theorie der linearen Differentialgleichungen. (Zweiter Teil)}. (German)
 J. Reine Angew. Math. 168 (1932), 233-252.


\bibitem{P66} J.-C. Petit, \emph{Sur certains quasi-corps g\'{e}n\'{e}ralisant un type d'anneau-quotient}.
S\'{e}minaire Dubriel. Alg\`{e}bre et Th\'{e}orie des Nombres 20 (1966 - 67), 1-18.



\bibitem{P17} S. Pumpl\"un, \emph{Nonassociative differential extensions of characteristic $p$}.
Results in Mathematics 72 (1-2) (2017), 245-262.\\
\verb#https://doi:10.1007/s00025-017-0656-x#


\bibitem{P16}  S. Pumpl\"un, \emph{Algebras whose right nucleus is a central simple algebra}.
Journal of Pure and Applied Algebra 222 (9),  2018, 2773-2783.\\
\verb#https://doi.org/10.1016/j.jpaa.2017.10.019#

\bibitem{P21}  S. Pumpl\"un, \emph{The automorphisms of generalized cyclic Azumaya algebras.}
J. Pure Applied Algebra 225 (4) (2021).\\
\verb#https://doi.org/10.1016/j.jpaa.2020.106540#


\bibitem{Psemi}  S. Pumpl\"un, \emph{Nonassociative cyclic algebras and the semiassociative Brauer monoid} (2024).\\
\verb#http://arxiv.org/abs/2403.03263#

\bibitem{Sch} R. D. Schafer, ``An Introduction to Nonassociative Algebras.'' Dover Publ. Inc., New York, 1995.

\bibitem{W09} G. P.~Wene, \emph{Inner automorphisms of finite semifields}. Note Mat. 29 (2009), suppl. 1, 231-242.\\
 \verb#https://doi.org/10.1285/i15900932v29n1supplp231#

\end{thebibliography}
\end{document}